\documentclass[11pt]{article}

\usepackage{amsmath,amsthm,amsfonts,amssymb,amscd, amsxtra}

\newcommand{\banacha}{\mathbb X}
\newcommand{\banachb}{\mathbb Y}
\newcommand{\norm}[1]{\|#1\|}
\newcommand{\Norm}[1]{\left\|#1\right\|}

\newcommand{\tildetheta}{\Theta}

\numberwithin{equation}{section}

\newtheorem{theorem}{Theorem}[section]
\newtheorem{lemma}[theorem]{Lemma}
\newtheorem{definition}[theorem]{Definition}
\newtheorem{corollary}[theorem]{Corollary}
\newtheorem{proposition}[theorem]{Proposition}

\newtheorem{remark}[theorem]{Remark}

\begin{document}
\title{A robust Kantorovich's theorem on inexact Newton method with relative residual error tolerance}
  
\author{ O. P. Ferreira\thanks{IME/UFG, Campus II- Caixa Postal 131,
    CEP 74001-970 - Goi\^ania, GO, Brazil (e-mail:{\tt
      orizon@mat.ufg.br}).  The author was supported in part by
    FUNAPE/UFG, CNPq Grant 473756/2009-9, CNPq Grant 302024/2008-5,
    PRONEX--Optimization(FAPERJ/CNPq) and IMPA.}  \and
  B. F. Svaiter\thanks{IMPA, Estrada Dona Castorina, 110, Jardim
    Bot\^anico, CEP 22460-320 - Rio de Janeiro, RJ, Brazil
    (e-mail:{\tt benar@impa.br}). The author was supported in part by
    CNPq Grant 301200/93-9(RN), CNPq Grant 475647/2006-8 and by
    PRONEX--Optimization(FAPERJ/CNPq).}  }
    
\date{ November 12, 2010 }

\maketitle
\begin{abstract}
We prove that under semi-local assumptions, the inexact Newton method with a
\emph{fixed} relative residual error tolerance converges $Q$-linearly to a zero
of the non-linear operator under consideration.
Using this result we show that Newton method for minimizing a self-concordant
function or to find a zero of an analytic function can be implemented with a
fixed relative residual error tolerance.

In the absence of errors, our analysis retrieve the classical Kantorovich
Theorem on Newton method.


\end{abstract}

\section*{Introduction}\label{sec:int}
Newton's method and its variations, including the inexact Newton
methods, are the most efficient methods known for solving nonlinear
equations
\[
F(x)=0,
\]
where $\banacha$ and $\banachb$ are Banach spaces, $C\subseteq\banacha$ and
$F:{C}\to \banachb$ is continuous and continuously differentiable on $int(C)$. 

Kantorovich's Theorem on Newton's method uses semi-local assumption on $F$ to
guarantee existence of a solution of the above equation, uniqueness of this
solution in a prescribed region and also convergence of Newton's Method to such
a solution, see \cite{kantorovich1951,kantAkil1964}.
Semi-local convergence theorems for Newton method has been instrumental in the
modern complexity analysis of the solution of polynomial (or analytical)
equations \cite{BlunCukerShubSmale1989, Smale1986}, linear and quadratic
programming problems and linear semi-definite programming problems
\cite{NesterovNemirovskii1994, Potra2005}.
These convergence results has also been used in the design and convergence analysis of
algorithms for these problems.
In all these applications, homotopy methods are combined with Newton's method,
which helps the algorithm to keep track of the solution of a parametrized
perturbed version of the original problem. Each Newton iteration requires the
solution of a linear system, and this accounts mostly for the computational
burden of these algorithms.

Since linear system are solved always inexactly in floating point
computations, it is natural to investigate robustness of Kantorovich's
and Kantorovich's-like
theorems under errors in the computation of the Newton step.  Moreover,
modern implementations of the conjugate gradients, coupled with
preconditioning, allows for the approximated solution of large linear
systems. It would be most desirable to have an \emph{a priori}
prescribed residual error tolerance in the iterative solutions of
linear system for computing the Newton steps, because this would prevent
over-solving and/or under-solving the linear system in question.
Although the local convergence analysis of Newton's method with relative
errors in the residue~\cite{Chen2006, Dembo1982, Morini1999} or in the
steep~\cite{Ypma1984} are well understood, the convergence analysis of the method
under general semi-local assumptions
assuming \emph{only} bounded relative residual errors
is a new contribution of this paper.
Previous works on this subject
include ~\cite{Moret1989, Shen2009}.  The advantage of working with an
error tolerance on the residual rests in the fact that the exact
Newton step need not to be know for evaluating this error, which makes
this criterion attractive for practical applications.

Recently, Kantorovich's theorem on Newton's Method was extended to
Riemannian manifolds using a new technique which simplifies the
analyses  and proof of this theorem, see \cite{FerreiraSvaiter2002}.
After that, this technique was successfully employed for proving generalized
versions of Kantorovich's theorem in Riemannian Manifolds and also in
the analysis of the classical version of Kantorovich's theorem in Banach
spaces, see\cite{Alvarez2008, FerreiraSvaiter2009, Li2006, Li2008, 
Li2009, Wang2009, Wang2006, Wang2007978}.
In the present work, we will use the technique introduced in
\cite{FerreiraSvaiter2002} to present a robust version of the
Kantorovich's theorem on the inexact Newton method with residual
relative error. It is worth to point out  that, for null error tolerance the analysis 
presented merge in the usual semi-local convergence analysis on Newton's method, 
see~\cite{FerreiraSvaiter2009}. The basic idea is to find
good regions for the inexact Newton method.  In these regions, the
majorant function bounds the non-linear function which root is
to be found, and the behavior of the inexact Newton iteration in these
regions is estimated using iterations associated to the majorant
function.  Moreover, as a whole, the union of all these regions is
invariant under inexact Newton's iteration.

This paper is organized as follows.
In Section \ref{sec:int.2}, some definitions and auxiliary results are
presented.
In Section \ref{sec:rerror} the main result is stated and some properties of
the majorant function are established.
The main relationships between the majorant function and the nonlinear operator
used in the paper are presented in Section~\ref{sec:br}.
In Section~\ref{sec:inpso} a family of regions where the behavior of the
inexact Newton iteration is estimated using the majorant function is
introduced. We also show that the union of all these regions is invariant under
the inexact Newton iteration with a fixed relative residual error tolerance..
In Section~\ref{sec:cinmre} the main result is proved.
In Section \ref{sec:scinmer}  we show that Newton method for minimizing a
self-concordant function under the usual semi-local assumption for these
functions, can be implemented with a fixed residual error tolerance.
Moreover, we show that Newton method for finding a zero of an analytic
function, under the usual semi-local assumption of the $\alpha$-theory can be
also be implemented with a fixed relative residual error tolerance.

\section{Basics definitions and auxiliary results} \label{sec:int.2}

Let $\banacha$  be a Banach space. The open and closed ball at $x$ are denoted, respectively by
\[
B(x,r) = \{ y\in \banacha ;\; \|x-y\|<r \},\qquad B[x,r] =
\{ y\in \banacha ;\; \|x-y\|\leqslant r \}.
\]
The following auxiliary results of elementary convex analysis will be needed:

\begin{proposition}
  \label{pr:conv.aux1}
Let $I\subset \mathbb{R}$ be an interval, and $\varphi:I\to \mathbb{R}$ be convex.
\begin{enumerate}
\item
For any $u_0\in \mathrm{int}(I)$, 
the application
\[ u\mapsto \frac{\varphi(u_0)-\varphi(u)}{u_0-u},\qquad  u\in I, u\neq u_0,\]
is increasing and  there exist (in $\mathbb{R}$)
$$
D^- \varphi(u_0)={\lim}_{u\to u_0 ^-} \; \frac{\varphi(u_0)-\varphi(u)}{u_0-u}
={\sup}_{u<u_0} \;\frac{\varphi(u_0)-\varphi(u)}{u_0-u}. \\
$$
\item If $u,v,w\in I$, $u<w$, and $u\leq v\leq w$ then
$$  \varphi(v)-\varphi(u) \leq \left[\varphi(w)-\varphi(u)\right]  \frac{v-u}{w-u}.
$$
\end{enumerate}
\end{proposition}

\begin{proof}
See \cite{HiriartLemarechal1993}.
\end{proof}

\begin{proposition}
\label{pr:nwt.imp}
If $h:[a,b) \to \mathbb{R}$ is convex, differentiable at $a$, $h'(a)<0$ and
\[\lim_{t\to b_{-}} h(t)=0,\]
then
\[ a-\frac{h(a)}{h'(a)}\leq b,\]
with equality if and only if $h$ is affine in $[a,b)$.
\end{proposition}

\begin{proof}
  Since $h$ is convex, $h(a)+h'(a)(t-a)\leq h(t)$ for any $t\in[a,b)$.
  Taking the limit $t\to b_{-}$ we obtain
  \[ 
    h(a)+h'(a)(b-a)\leq 0.
  \]
  The desired inequality now follows multiplying this inequality by the strictly positive number $-1/h'(a)$. If the above inequality holds as an equality,  then
  \[
   h'(a)=\frac{-h(a)}{b-a}.
  \]
  Let $a\leq s< t<b$. Using again the convexity of $h$ we have  \[ h(a)+h'(a)(s-a)\leq h(s)\leq h(a)\frac{t-s}{t-a}+h(t)\frac{s-a}{t-a}.
  \]
  Taking again the limit $t\to b_{-}$ in the above equation and using
  the previous equation we conclude that $h(s)=h(a)(b-s)/(b-a)$, i.e.,
  $h$ is affine. If $h$ is affine then the the inequality
  of the proposition holds trivially as an equality.
\end{proof}

\section{The inexact Newton method with relative error} \label{sec:rerror}

Our goal is to prove the following version of
Kantorovich's theorem on inexact Newton's Method with relative error.

\begin{theorem}
  \label{th:ki.r}
  Let $\banacha$ and $\banachb$ be Banach spaces,  $R\in \mathbb{R}$, $C\subseteq
  \banacha$ and $F:{C}\to \banachb$ a continuous function,
  continuously differentiable on $int(C)$.  Take $x_0\in
  \mathrm{int}(C )$ with $F '(x_0)$ non-singular.  Suppose that
  \mbox{$ f:[0,R)\to\mathbb{R}$} is continuously differentiable,
  $B(x_0,R)\subseteq C$,
  \begin{equation}\label{MC.2}
  \|F'(x_0)^{-1}\left[F'(y)-F'(x)\right]\| \leq
    f'(\|y-x\|+\|x-x_0\|)-f'(\|x-x_0\|),
  \end{equation}
   for any $x,y\in B(x_0,R)$,  $\|x-x_0\|+\|y-x\|< R$,
  \begin{equation}    \label{eq:KH.2}
    \|F'(x_0)^{-1}F(x_0)\|\leq f(0)\,,
  \end{equation}
  \begin{itemize}
    \item[{\bf h1)}] $f(0)>0$, $f'(0)=-1$;
    \item[{\bf h2)}] $f'$ is strictly increasing and convex;
    \item[{\bf h3)}] $f(t)<0$ for some $t\in (0,R)$.
  \end{itemize}
  Let 
  \[
   \beta:=\sup_{t\in[0,R)} -f(t),\;\;  t_*:=\min f^{-1}(\{0\}),
   \;\; \bar \tau:=\sup\{t\in [0,R)\,:\, f(t)<0\}.
  \]
  Take $0\leq \rho<\beta/2$ and define
  \begin{equation} \label{eq:dktt}
  \kappa_\rho:=\sup_{\rho <t<R}\frac{-(f(t)+2\rho)}{|f'(\rho)|\,(t-\rho)},\quad
  \lambda_\rho:=\sup \{t\in [\rho,R): \kappa_\rho+f'(t)<0\},
  \quad 
  \tildetheta_\rho:=\frac{\kappa_\rho}{2-\kappa_\rho}.
  \end{equation}
  Then for any $\theta\in [0,\tildetheta_\rho]$ and $z_0\in B(x_0,\rho)$,
  the sequence generated by the inexact Newton method for solving
  $F(x)=0$ with starting point $z_0$ and residual relative error
  tolerance $\theta$: For $k=0,1,\ldots,$
  \[
    z_{k+1}=z_k+S_k,\quad   
    \Norm{F'(z_0)^{-1}\left[F(z_k)+F'(z_k)S_k\right]}\leq \theta
    \|F'(z_0)^{-1}F(z_k)\|,
  \]
  is well defined (for any particular choice of each $S_k$), 
\begin{equation}
  \label{eq:th.rt}
 \|F'(z_0)^{-1}F(z_k)\|\leq
  \left(\frac{1+\theta^2}{2}\right)^{k}[f(0)+2\rho], \qquad k= 0,1,
  \ldots \,,
\end{equation}
 the  sequence $\{z_k\}$ is contained in
  $B(z_0, \lambda_\rho )$ and converges to a point $x_*\in B[x_0, t_*]$,
  which is the unique zero of $F$ on $B(x_0,\bar \tau)$. Moreover, if
  \begin{itemize}
  \item[{\bf h4)}] $\lambda_\rho<R-\rho$,
  \end{itemize}
 then the sequence $\{z_k \}$ satisfies, for $k= 0,1, \ldots \,$, 
\[
\|x_*-z_{k+1}\|\leq \left[ \frac{1+\theta}{2}\frac{D^{-}f'(\lambda_\rho)}{|f'(\lambda_{_\rho})|}\|x_*-z_k\|
+\theta\,\frac{f'(\lambda_{_\rho}+\rho)+2|f'(\rho)|}{|f'(\lambda_{_\rho}+\rho)|}\right]\|x_*-z_k\|.
\]
  If, additionally, $0\leq \theta <\kappa_\rho/(4+\kappa_\rho)$ then  $\{z_k \}$ converges $Q$-linearly as  follows
  \[
  \norm{x_*-z_{k+1}}\leq \left[\frac{1+\theta}{2}+\frac{2\theta}{\kappa_\rho} \right]\norm{x_*-z_k}, \qquad k= 0,1,
  \ldots \,.
  \]
\end{theorem}

\begin{remark}
In Theorem~\ref{th:ki.r} if $\theta=0$ we obtain the  exact Newton method and
its convergence properties. Now, taking $\theta=\theta_k$  in each iteration
and letting $\theta_k$ goes to zero as $k$ goes to infinity, the  penultimate
inequality of the Theorem~\ref{th:ki.r} implies that the generated sequence
converges to the solution with superlinear rate.
\end{remark}

>From now on, we assume that the hypotheses of Theorem \ref{th:ki.r} hold.  The
scalar function $f$ in the above theorem is called a {\it majorant function}
for $F$ at point $x_0$. Before proceeding, we will analyze some basic
properties of the majorant function.  Condition {\bf h2} implies in strict
convexity of $f$. Note that  $t_*$ is the smallest root of $f(t)=0$ and, since
$f$ is convex, if this equation has two roots, then the second one is
$\bar\tau$.

 Define
\begin{equation}
  \label{eq:def.bart}
  \bar t:=\sup \{t\in [0,R): f'(t)<0\}\;.
\end{equation}

\begin{proposition}
  \label{pr:maj.f}
  The following statements on the majorant function hold
  \begin{itemize}
  \item[i)] $f'(t)<0$ for any $t\in [0,\bar t)$, (and
    $f'(t)\geq 0$ for any $t\in [0,R)\setminus [0,\bar t)$);
  \item[ii)] $0< t_* <\bar t\leq \bar \tau\leq R$;
  \item[iii)] $ \beta=-\lim_{t\to \bar t_{-}} f(t),\qquad 0< \beta <\bar t$.
\end{itemize}
\end{proposition}

\begin{proof}
  Item {\it i} follows from the second part of  {\bf h1},  {\bf h2} and
  the definition \eqref{eq:def.bart}.

  Using the first inequality in {\bf h1}, {\bf h3} and the continuity
  of $f$ we conclude that $t_*$ is well defined and
  \[
  0< t_*<R.
  \]
  Condition {\bf h2} implies in strict convexity of $f$, hence
  condition {\bf h3} and the definition of $t_*$ imply that there exists
  $t\in (t_*, R)$ such that
  \[
  0>f(t)>f(t_*)+f'(t_*)(t-t_*)=f'(t_*)(t-t_*),
  \]
  which implies that $0>f'(t_*)$. Therefore, using item i and the
  definition of $\bar t$ we have
  \[
  t_* <\bar t\leq R.
  \]
  Since $t_*$ is the smallest root of $f(t)=0$ and $f$ is strictly
  decreasing in $[0, \bar t)$ we conclude that $f<0$ in $[t_*, \bar
  t)$. So, the definition of $\bar \tau$ implies that
  \[
  \bar t\leq \bar \tau\leq R,
  \]
  and the proof of item ii is concluded.

  Using {\bf h3} and the definition of $\beta$ we obtain that $ 0<\beta.  $ Since $f$ is convex, combining this with {\bf h1} we have
  \[
  f(t)\geq f(0)-t>-t, \qquad 0\leq t<R,
  \]
  with strict inequality for $t\neq 0$. We know that $f$ is strictly
  decreasing and $f<0$ in $[t_*, \bar t)$. Hence, letting $t$ goes to
  $ \bar t_{-}$ in last inequality and using the definition of $\beta$ the
  item iii follows.
\end{proof}
We will first prove Theorem~\ref{th:ki.r} for the case $\rho=0$ and $z_0=x_0$. 
In order to  simplify the notation in the case $\rho=0$, we will use
$\kappa$, $\lambda$ and $\theta$ instead of $\kappa_0$, $\lambda_0$
and $\theta_0$ respectively:
\begin{equation} \label{eq:dktt.0}
  \kappa:=\sup_{0<t<R}\frac{-f(t)}{t},\quad
  \lambda:=\sup \{t\in [0,R): \kappa+f'(t)<0\},
  \quad 
  \tildetheta:=\frac{\kappa}{2-\kappa}.
\end{equation}
\begin{proposition}
  \label{pr:new}
  For $\kappa,\lambda,\theta$ as in \eqref{eq:dktt.0} it holds that 
\begin{equation}
  \label{eq:bd.pr.a}
   0<\kappa<1,\qquad  0<\tildetheta<1,\qquad  t_*<\lambda\leq\bar t \leq\bar \tau,
\end{equation}
and 
\begin{equation}
  \label{eq:bd.pr.b}
  \begin{array}{c}
    \displaystyle  
    f'(t)+\kappa<0, \qquad \forall\; t\in[0,\lambda),\\[.7em]
    \displaystyle    \inf_{0\leq t<R}  f(t)+\kappa t=\lim_{t\to\lambda_{-}}  f(t)+\kappa t=0,
  \end{array}
\end{equation}
\end{proposition}
\begin{proof}
Since $f$ is convex, combining this with {\bf h1} we have
\[
  f(t)\geq f(0)-t>-t, \qquad 0\leq t<R,
\]
with strict inequality for $t\neq 0$. For $t\neq 0$, last inequality is equivalent to
\[
 \frac{-f(t)}{t}\leq 1-\frac{f(0)}{t}<1-\frac{f(0)}{R}<1,\qquad 0<t< R,
\]
and, using also {\bf h3}, we conclude that 
\[
   0<\kappa<1,\qquad  0<\tildetheta<1,
\]
where the bounds on $\tildetheta$ follows from its definition and the
bound on $\kappa$.  Moreover, as $f'$ is continuous, strictly increasing and $f'(0)=-1$, we obtain
\[
  \begin{array}{c}
    \displaystyle  
    0<\lambda,\qquad 
    f'(t)+\kappa<0, \qquad  \forall\; t\in [0, \lambda),\\[.7em]
    \displaystyle    \inf_{0\leq t<R}  f(t)+\kappa t=\lim_{t\to\lambda_{-}}  f(t)+\kappa t=0,
  \end{array}
\]
where the last equalities  follows from the
definition of~$\kappa$. 

Note that $f'(t)=-\kappa<0$ for all $t\in [0, \lambda)$. Since $f'$ is strictly negative in
$[0,\lambda)$, we conclude that $ t_*<\lambda\leq\bar t\leq\bar \tau$ and the proof is concluded.
\end{proof}
\section{Basic results} \label{sec:br}
In this section we will obtain bounds on $\|F'^{-1}\|$ and on the
linearization error on $F$ using the majorant function $f$. This
bounds will be used in the next section for analyzing the inexact Newton
iterations. It is worth mentioning that in this section the 
inequality on {\bf h1} and \eqref{eq:KH.2} will be used only
for proving its last result and {\bf h3} will not be used.

A Newton iteration at $x$ requires non-singularity of $F'(x)$, which will be
verified using the majorant function $f$.

\begin{proposition}\label{pr:banach}
  If \,\,$\| x-x_0\|\leq t<\bar{t}$ then $F'(x) $ is non-singular and
  \[
   \|F'(x)^{-1}F'(x_0)\|\leq  \frac{1}{-f'(t)}.
  \]
\end{proposition}
\begin{proof}
  The definition \eqref{eq:def.bart} shows that $f'(t)<0$. Direct
  manipulation, \eqref{MC.2}, {\bf h1} and {\bf h2} give us
\begin{align*}
        \|F'(x_0)^{-1}F'(x)-I\|= \|F'(x_0)^{-1}[F'(x)-F'(x_0)]\|&\leq
   f'(\|x-x_0\|)-f'(0)\\ &=  f'(t)+1 <1.
\end{align*}
Using   Banach's Lemma and the last inequality above we conclude that 
$F'(x_0)^{-1}F'(x)$ is non-singular and
\[ 
\|F'(x)^{-1}F'(x_0)\|=\|(F'(x_0)^{-1}F'(x))^{-1}\|
\leq  \frac{1}{1-(f'(t)+1)},
\]
which is the desired inequality.
\end{proof}

The linearization errors on $F$ and $f$ are, respectively
\begin{align}
  \label{eq:def.er}
  E_F(y,x):=& F(y)-\left[ F(x)+F'(x)(y-x)\right],
             &x&\in
  B(x_0, R),&y&\in C\\
\label{eq:def.ers}
    e_f(v,t):=&f(v)-[f(t) +f'(t)(v-t)], 
             &t&,v\in [0, R). &&
\end{align}
The linearization error of the majorant function bounded the
linearization error of $F$.

\begin{lemma} 
 \label{pr:taylor}
  If $x,y\in \banacha$ and $\norm{x-x_0}+\norm{y-x} < R$ then
  \[
  \|F'(x_0)^{-1}E_F(y,x)\|\leq
  e_f(\norm{x-x_0}+\norm{y-x},\norm{x-x_0})\,, 
  \]
\end{lemma}
\begin{proof}
  Since  
  \[
   x+u(y-x)\in B(x_0, R), \quad \quad 0\leq u\leq 1,
  \]
  and $F$ is continuously differentiable in $B(x_0, R)$, direct
  use of \eqref{eq:def.er} gives
  \[ 
  E_F(y,x)=\int_0^1 [F'(x+u(y-x))-F'(x)](y-x)\; du.
  \]
  Combining the above equality with \eqref{MC.2}  we have
  \begin{align*}
    &\|F'(x_0)^{-1}E_F(y,x)\| \\
    &\leq \int_0 ^1 \left \|
    F'(x_0)^{-1}[F'(x+u(y-x))-F'(x)]\right\|\left\|y-x\right\| \; du\\ 
    &\leq \int_0 ^1 \left[f'\left(\|x-x_0\|+u\left\|y-x\right\|\right)-f'\left(\|x-x_0\|\right)
    \right] \|y-x\|\;du\,
  \end{align*}
  which after performing the integration and using the definition in
  \eqref{eq:def.ers} yields the desired inequality.
\end{proof}
Convexity of $f$ and $f'$ guarantee that $e_f(t+s,t)$ is increasing in
$s$ \emph{and} $t$.
\begin{lemma} \label{l:errormon}
If  $0\leq b\leq t$, $0\leq a\leq s$ and $t+s<R$ then
\begin{align*}
 e_f(a+b,b)&\leq e_f(t+s,t),\\
 e_f(a+b,b)&\leq   \frac 1 2 \frac{f'(t+s)-f'(t)}{s}\;a^2, \qquad s\neq 0.
\end{align*}
\end{lemma}
\begin{proof}
  First note that
\[ e_f( a+b,b)= \int_0 ^{a} 
    \left[f'\left(b+ r\right)-f'\left(b\right)
    \right]
   d r\,.
 \]
 Since $f'$ is convex, for any $\tau_0>0$, the function $\tau \mapsto
 f'(\tau+\tau_0)-f'(\tau)$ is non-decreasing. So,
\begin{equation} \label{eq:ber}
 e_f( a+b,b)\leq \int_0 ^{a} 
    \left[f'\left(t+ r\right)-f'\left(t\right)
    \right]
   d r\,
\leq \int_0 ^{s} 
    \left[f'\left(t+ r\right)-f'\left(t\right)
    \right]
   d r\,.
 \end{equation}
 where the second inequality follows from the convexity of $f$, which
 implies positivity of the integrand.  To end the proof of first inequality, note that the
 last term on the above inequality is $e_f(t+s,t)$.
 
 For proving second inequality,  apply Proposition~\ref{pr:conv.aux1} item 2  with $u=t$, $v=t+r$,  $w=t+s$ and $\varphi=f'$ in first inequality in \eqref{eq:ber} to conclude that
 \[
 e_f( a+b,b)\leq \int_0 ^{a} [f'(t+s)-f'(t)] \,\frac{r}{s}\, dr,
 \]
which  performing the integral gives the desired inequality.
\end{proof}
Now we are ready to bound the  linearization error $E_F$ 
using  the linearization error on 
the majorant function.
\begin{corollary} \label{pr:taylor2}
If $x, y\in  \banacha$,  $\norm{x-x_0}\leq t$,  $\norm{y-x}\leq s$ and $s+t<R$ then
\begin{align*}
\|F'(x_0)^{-1}E_F(y,x)\| &\leq e_f(t+s,t),\\
\|F'(x_0)^{-1}E_F(y,x)\|&\leq \frac{1}{2} \frac{f'(s+t)-f'(t)}{s}\norm{y-x}^{2}, \qquad  s\neq 0. 
\end{align*}
\end{corollary}
\begin{proof}
The results follows by direct combination of the Lemmas \ref{pr:taylor} and \ref{l:errormon} by taking $b=\norm{x-x_0}$ and $a=\norm{y-x}$.
\end{proof}
The first inequality in the next corollary will be useful for obtaining  
asymptotic bounds on the sequence generated by the inexact Newton
method with relative error tolerance, while the second inequality will
be used to show that this method is robust with respect to the
initial iterate.
\begin{corollary}
  \label{cr:new.01}
  For any $y\in B(x_0,R)$,
  \[ -f(\Norm{y-x_0})\leq \Norm{F'(x_0)^{-1}F(y)}
   \leq f(\Norm{y-x_0})+2\Norm{y-x_0}.
 \]
\end{corollary}
\begin{proof}
  Using Lemma~\ref{pr:taylor} with $x=x_0$,   the definition of $E_F$
 and  triangle inequality
 we have
  \begin{align*}
    e_f(\Norm{y-x_0},0)\geq&\Norm{F'(x_0)^{-1}E_F(y,x_0)}\\
      \geq&\Norm{ F'(x_0)^{-1}F(x_0)+y-x_0}
       -\Norm{F'(x_0)^{-1}F(y)}\\
      \geq&\Norm{y-x_0}-\Norm{F'(x_0)^{-1}F(x_0)}
       -\Norm{F'(x_0)^{-1}F(y)}.
  \end{align*}
  Combining this inequality with the definition of $e_f$ and using the
  assumption {\bf h1} and \eqref{eq:KH.2}
  we obtain
  \[ 
  f(\Norm{y-x_0})-f(0)+\Norm{y-x_0}\geq
  \Norm{y-x_0}-f(0)-\Norm{F'(x_0)^{-1}F(y)},
  \]
  which is equivalent to the firs inequality of the corollary.

To prove the second inequality, use again 
Lemma~\ref{pr:taylor}  the definition of $E_F$
 and  triangle inequality to obtain
 \begin{align*}
    e_f(\Norm{y-x_0},0)\geq&\Norm{F'(x_0)^{-1}E_F(y,x_0)}\\
      \geq& \Norm{F'(x_0)^{-1}F(y)}- 
      \Norm{F'(x_0)^{-1}F(x_0)+y-x_0}\\
      \geq&\Norm{F'(x_0)^{-1}F(y)}
      -\Norm{F'(x_0)^{-1}F(x_0)}-\Norm{y-x_0}.
  \end{align*}
Using the above inequality, the definition of $e_f$,   {\bf h1} and \eqref{eq:KH.2} we have
\[ f(\Norm{y-x_0})-f(0)+\Norm{y-x_0}\geq\Norm{F'(x_0)^{-1}F(y)}
      -f(0)-\Norm{y-x_0}.
\]
 which is equivalent to the second inequality of the corollary.
\end{proof}
Note that the first inequality on the above corollary proves that $F$
has no zeroes in the region $t_*< \Norm{x-x_0}<\bar \tau$.
\begin{lemma} \label{l:bd}
If $x\in \banacha$, $\norm{x-x_0}\leq t< R$ then
\[
\|F'(x_0)^{-1}F'(x)\|\leq 2+f'(t).
\]
\end{lemma}
\begin{proof}
Simple algebraic manipulation together with assumption \eqref{MC.2} give us
\[
\|F'(x_0)^{-1}F'(x)\|\leq I+F'(x_0)^{-1}[F'(x)-F'(x_0)]\leq 1+f'(\norm{x-x_0})-f'(0).
\]
Hence,    {\bf h1}, {\bf h2}  and the last inequality imply the statement of the lemma.
\end{proof}
\begin{lemma}  \label{pr:cq1}
Take $\theta\geq 0$, $0\leq t\leq\lambda$, $x_*, x, y\in  \banacha$. If  $\lambda<R$, 
$\norm{x-x_0}\leq t$,  $\norm{x_*-x}\leq \lambda-t$, $F(x_*)=0$ and
\begin{equation} \label{eq:y.in1}
\Norm{F'(x_0)^{-1}[F(x)+F'(x)(y-x)]}\leq \theta \norm{F'(x_0)^{-1}F(x)},
\end{equation} 
  then
\begin{align}
\|x_*-y\|&\leq \left[ \frac{1+\theta}{2}
+\frac{2\theta}{\kappa}\right]\|x_*-x\|,\\
\|x_*-y\|&\leq \left[ \frac{1+\theta}{2}\frac{D^{-}f'(\lambda)}{|f'(\lambda)|}\|x_*-x\|
+\theta\,\frac{2+f'(\lambda)}{|f'(\lambda)|}\right]\|x_*-x\|.
\end{align}
\end{lemma}
\begin{proof}
Since $F(x_*)=0$, direct algebraic manipulation and \eqref{eq:def.er}  yield
\[
y-x_*=F'(x)^{-1}\left[E_F(x_*, x)+[F(x)+F'(x)(y-x)]\right].
\]
Using \eqref{eq:y.in1},  properties of the norm and some simple  manipulations we conclude from last equality  that 
\[
\|x_*-y\|\leq \Norm{F'(x)^{-1}F'(x_0)}\left[\Norm{F'(x_0)^{-1}E_F(x_*, x)}+\theta \Norm{F(x_0)^{-1}F(x)} \right].
\]
On the other hand, using again $F(x_*)=0$ and the definition in \eqref{eq:def.er} we have 
\[
-F'(x_0)^{-1}F(x)=F'(x_0)^{-1}\left[E_{F}(x_*, x)+F'(x)(x_*- x)\right],
\]
which using the triangular inequality yields
\[
\Norm{F'(x_0)^{-1}F(x)}\leq \Norm{F'(x_0)^{-1}E_{F}(x_*, x)}+ \Norm{F'(x_0)^{-1}F'(x)}\Norm{x_*-x}.
\]
Combining two above inequalities with  Proposition~\ref{pr:banach}, Corollary~\ref{pr:taylor2} with $y=x_*$ and $s=\lambda-t$ and Lemma~\ref{l:bd} we have
\[
\|x_*-y\|\leq \frac{1}{|f'(t)|}\left[ \frac{1+\theta}{2} \frac{f'(\lambda)-f'(t)}{\lambda-t}\|x_*-x\|
+\theta\;[2+f'(t)]\right]\|x_*-x\|.
\]
Since $\|x_*-x\|\leq  \lambda-t$,  $f'<-\kappa<0$ in $[0, \lambda) $  and $f'$ is increasing the first inequality follows  from last inequality.

Using  Proposition~\ref{pr:conv.aux1} and taking in account that $f'<0$ in $[0, \lambda)$ and increasing  we obtain the second inequality from above inequality.
\end{proof}
\section{The inexact Newton iteration with relative error}\label{sec:inpso}
In the next lemma we study a single inexact Newton iteration with 
relative error $\theta$.
\begin{lemma}
  \label{lm:bas} Take $t,\varepsilon,\theta\geq 0$, and $x\in C$ such that
  \begin{equation}
    \label{eq:ddks} \norm{x-x_0}\leq t<\bar t,\quad
    \norm{F'(x_0)^{-1}F(x)}\leq f(t)+\varepsilon, \quad
    t-(1+\theta)\frac{f(t)+\varepsilon}{f'(t)}< R.
  \end{equation} 
  If $y\in X$ and
  \begin{equation}
    \label{eq:y.in} \norm{F'(x_0)^{-1}[F(x)+F'(x)(y-x)]}\leq \theta
\norm{F'(x_0)^{-1}F(x)}.
\end{equation} then
\begin{enumerate}
\item \( \|y-x\|\leq \displaystyle
  -(1+\theta)\frac{f(t)+\varepsilon}{f'(t)}; \)
\item \( \|y-x_0\|\leq \displaystyle
  t-(1+\theta)\frac{f(t)+\varepsilon}{f'(t)}<R; \)
\item 
  \( \|F'(x_0)^{-1}F(y)\| \leq \displaystyle f \left(t- (1+\theta)
    \frac{f(t)+\varepsilon}{f'(t)}\right) +
  \varepsilon+2\theta(f(t)+\varepsilon).  \)
\end{enumerate}
\end{lemma}
\begin{proof} 
Using Proposition \ref{pr:banach} and the first inequality in
\eqref{eq:ddks} we conclude that $F'(x)$ is non-singular and
$\Norm{F'(x)^{-1}F'(x_0)}\leq -1/f'(t)$. Therefore, using also the identity
 \begin{align*} y-x
&=F'(x)^{-1}F'(x_0)\bigg[
F'(x_0)^{-1}\left[F(x)+F'(x)(y-x)\right]
-F'(x_0)^{-1}F(x)\bigg],
\end{align*} 
triangular inequality and \eqref{eq:y.in}
we conclude that 
\[
\norm{y-x}\leq \frac{-1}{f'(t)}(1+\theta) \Norm{F'(x_0)^{-1}F(x)}\;.
\]
To end the proof of item 1, use the above inequality and the second
inequality on \eqref{eq:ddks}.

Item 2 follows from triangular inequality, item 1 and   the first
and the third inequalities in \eqref{eq:ddks}.

Using the definition of the error \eqref{eq:def.er} we
have
\[ F(y)=E_F(y,x)+F'(x_0)\bigg[
F'(x_0)^{-1}[F(x)+F'(x)(y-x)]
\bigg].
\]
Therefore, using  the triangle inequality,  \eqref{eq:y.in} 
and the second inequality on \eqref{eq:ddks}
 we have
 \begin{align*}
    \norm{F'(x_0)^{-1}F(y)}\leq& \norm{F'(x_0)^{-1}E_F(y,x)}+\theta
 \Norm{F'(x_0)^{-1}F(x)}\\
\leq&
 \norm{F'(x_0)^{-1}E_F(y,x)}+\theta(f(t)+\varepsilon)  .
 \end{align*}
Using \eqref{eq:ddks}, item 1, and Lemma~\ref{pr:taylor} with
$s=-(1+\theta)(f(t)+\varepsilon)/f'(t)$ we have
\begin{align*}
   \norm{F'(x_0)^{-1}E_F(y,x)}\leq&
e_f\left(t-(1+\theta)\frac{f(t)+\varepsilon}{f'(t)}, t\right)\\
=&f\left(t-(1+\theta)\frac{f(t)+\varepsilon}{f'(t)}\right)
+\varepsilon+
\theta ( f(t)+\varepsilon).
\end{align*}
Direct combination of the two above equation  yields the latter
inequality in item 3.
\end{proof} 
In view of Lemma~\ref{lm:bas} define, for $\theta\geq 0$, the auxiliary map
$n_{\theta}: [0,\bar t\,)\times [0,\infty) \to \mathbb{R}\times
\mathbb{R}$,
\begin{equation} \label{eq:nftheta}
  \displaystyle n_{\theta}(t,\varepsilon):
  =\left(t-(1+\theta)\frac{f(t)+\varepsilon}{f'(t)},\;\varepsilon+
 2\theta(f(t)+\varepsilon)\right).
\end{equation}
Let 
\begin{equation}
  \label{eq:dt} {\Omega}:=\left \{
(t,\varepsilon)\in\mathbb{R}\times\mathbb{R}\;:\; 0\leq t<\lambda,
  \;\; 0\leq\varepsilon\leq \kappa t,\;\; 0<f(t)+\varepsilon \right\}.
\end{equation}
\begin{lemma}
  \label{lm:id.t}
  If $0\leq \theta\leq \tildetheta$, $(t,\varepsilon)\in{\Omega}$ 
  and $(t_+,\varepsilon_+)=n_{\theta}(t,\varepsilon)$, that is,
  \[
  t_+=t-(1+\theta)\frac{f(t)+\varepsilon}{f'(t)}, \qquad
  \varepsilon_+=\varepsilon+2\theta (f(t)+\varepsilon),
  \]
  then $n_\theta(t,\varepsilon)\in {\Omega}$, $t<t_+$, $ \varepsilon \leq
  \varepsilon_+$ and
  \[
  f(t_+)+\varepsilon_+<\left(\frac{1+\theta^2}{2}\right)(f(t)+\varepsilon).
  \]
\end{lemma}
\begin{proof}
 Since $ 0\leq t<\lambda$,   according to \eqref{eq:dktt.0} we have
  $f'(t)<-\kappa<0$. Therefore $t<t_+$ and $\varepsilon\leq\varepsilon_+$.
  
As $\varepsilon\leq\kappa t$, $f(t)+\varepsilon>0$ and $-1\leq f'(t)<f'(t)+\kappa<0$,
\begin{align}
\nonumber
 -\frac{f(t)+\varepsilon}{f'(t)}\leq& 
 -\frac{f(t)+\kappa \,t}{f'(t)}\\
 \label{eq:aux.01}
   =& -\frac{f(t)+\kappa \,t}{f'(t)+\kappa}\;\left[1+\frac{\kappa}{f'(t)}\right]\leq
    -\frac{f(t)+\kappa \,t}{f'(t)+\kappa}(1-\kappa).
\end{align}
The function $h(s):=f(s)+\kappa s$ is differentiable at $t$, $h'(t)<0$,
is strictly convex and 
\[ \lim_{s\to\lambda_{-}} h(t)=0.
\]
Therefore, using Proposition~\ref{pr:nwt.imp} we have $ t-h(t)/h'(t)<\lambda,$
which is equivalent to
\begin{equation}
  \label{eq:aux.02}
  -\frac{f(t)+\kappa \,t}{f'(t)+\kappa}< \lambda-t.  
\end{equation}
Combining the above inequality with \eqref{eq:aux.01} and the
definition of $t_+$  we conclude that
\[ t_+< t+(1+\theta)(1-\kappa)(\lambda-t).
\]
Using \eqref{eq:dktt.0} and \eqref{eq:bd.pr.a} we have $(1+\theta)(1-\kappa)\leq1-\theta<1$, which combined with the above
inequality yields $t_+<\lambda$.

  Using the definition of $\varepsilon_+$, inequality $\varepsilon
\leq \kappa\;t$ and \eqref{eq:dktt.0} we obtain
  \begin{align*} \varepsilon_+&\leq 2\theta(f(t)+\varepsilon)+\kappa\; t\\ &=\kappa
\left(t+(1+\theta)(f(t)+\varepsilon)\right).
  \end{align*} Using again the inequalities $f(t)+\varepsilon>0$ and
$-1\leq f'(t)<0$ we have
  \[f(t)+\varepsilon\leq -\frac{f(t)+\varepsilon}{f'(t)}.
  \] Combining the two above inequalities with the definition of $t_+$
we obtain $\varepsilon_+\leq \kappa\,t_+$.

  For  proving the two last inequalities, first note that from the definition
of the linearization error in \eqref{eq:def.ers} we have
\begin{align*} f(t_+)+\varepsilon_+&=
f(t)+f'(t)(t_+-t)+e_f(t_+,t)+2\theta(f(t)+\varepsilon)+\varepsilon\notag\\
&=\theta(f(t)+\varepsilon)+e_f(t_+,t) \\
&=\theta(f(t)+\varepsilon)+\int_{t}^{t_+}\left(f'(u)-f'(t)\right)du. \notag
\end{align*} 
Since $f'$ is strictly increasing we conclude that the integral is
positive. So, last equality implies that $ f(t_+)+\varepsilon_+\geq
\theta (f(t)+\varepsilon)>0$. 
Taking  $s\in [t_+,\lambda)$ and using the convexity of $f'$ we have
\begin{align*} 
\int_{t}^{t_+}\left(f'(u)-f'(t)\right) du &\leq
\int_{t}^{t_+}\left( f'(s)-f'(t)\right) \frac{u-t}{s-t} \;du \\
 &=\frac{1}{2}\frac{(t_+-t)^2}{s-t}\left(
f'(s)-f'(t)\right).
 \end{align*} 
Substituting last inequality into above equation we have
\begin{align*}
f(t_+)+\varepsilon_+&\leq\theta(f(t)+\varepsilon)
+\frac{1}{2}\frac{(t_+-t)^2}{s-t}\left( f'(s )-f'(t)\right)\\
&=\left(\theta+\frac{1}{2}\,\frac{(1+\theta)^2}{(s-t)}\,
\frac{f(t)+\varepsilon}{-f'(t)}\,\frac{f'(s)-f'(t)}{-f'(t)} \right)(f(t)+\varepsilon).
\end{align*}
On the other hand, because $f'(s )+\kappa<0$ and
$-1\leq f'(t)$ it easy to conclude that
\[
 \frac{f'(s )-f'(t)}{-f'(t)}=\frac{f'(s)+\kappa-f'(t)-\kappa}{-f'(t)}\leq 1-\kappa.
\]
Combining last two above inequalities with \eqref{eq:aux.01}, \eqref{eq:aux.02}  and taking
in account that   $(1+\theta)(1-\kappa)\leq1-\theta$ we
conclude that
\begin{align*}
 f(t_+)+\varepsilon_+\leq&
\left(\theta+\frac{1}{2}(1+\theta)^2(1-\kappa)^2\frac{\lambda-t}{s-t}\right)
(f(t)+\varepsilon)\\
=&\left(\theta+\frac{1}{2}(1-\theta)^2\frac{\lambda-t}{s-t}\right)(f(t)+\varepsilon),  
\end{align*}
and the result follows taking the limit $s\to\lambda_{-}$.
\end{proof}
The outcome of an inexact Newton iteration is any point satisfying
some error tolerance. Hence, instead of a mapping for Newton
iteration, we shall deal with a \emph{family} of mappings, describing
all possible inexact
iterations.
\begin{definition}
  \label{def:inire}
  For $0\leq \theta$, $\mathcal{N}_\theta$ is the family of maps
  $N_\theta:B(x_0,\bar t)\to X$ such that
  \begin{equation}
    \label{eq:in.map}
    \Norm{F'(x_0)^{-1}[F(x)+F'(x)(N_\theta(x)-x)]}\leq \theta
    \Norm{F'(x_0)^{-1}F(x)},
  \end{equation}
  for each $x\in B(x_0,\bar t\,)$.
\end{definition}
If $x\in B(x_0,\bar t)$, then $F'(x)$ is
non-singular. Therefore, for $\theta= 0$, the family $\mathcal{N}_0$ has a
single element, 
namely the exact Newton iteration map
\[N_0:B(x_0,\bar t)\to \banacha, \quad x\mapsto N_0(x)= x-F'(x)^{-1}F(x).
\]
Trivially, if $0\leq \theta \leq \theta '$ then 
$\mathcal{N}_0\subset\mathcal{N}_\theta\subset\mathcal{N}_{\theta '}$.
Hence, $\mathcal{N}_\theta$ is non-empty for all $\theta\geq 0$.
\begin{remark}
  For any $\theta \in (0,1)$ and $N_\theta\in\mathcal{N}_\theta$
  \[ N_\theta(x)=x\iff F(x)=0,\qquad x\in B(x_0,\bar t).
  \]
  This means that the fixed point of the inexact Newton iteration
  $N_\theta$ are the same fixed points of the \emph{exact} Newton
  iteration, namely,  the zeros of $F$.
\end{remark}
The main tool for the analysis of the inexact Newton method with a
relative residual tolerance will be a family of sets described below
and analyzed in the ensuing proposition, which is a combination of Lemmas
\ref{lm:bas} and \ref{lm:id.t}. Define
\begin{equation}
  \label{eq:ker} K(t,\varepsilon):=\left\{x\in X\;:\; \Norm{x-x_0}\leq
t,\;\Norm{F'(x_0)^{-1}F(x)}\leq f(t)+\varepsilon\right\},
\end{equation} 
and
\begin{equation}
  \label{eq:kt}  
K:=\bigcup_{(t,\varepsilon)\in {\Omega}}
  K(t,\varepsilon).
\end{equation}
Recall that $n_\theta$, $\Omega$ and $\mathcal{N}_\theta$ were defined in
\eqref{eq:nftheta}, \eqref{eq:dt} and Definition \ref{def:inire} respectively.

\begin{proposition}
  \label{pr:syn} 
   Take $0\leq \theta\leq \tildetheta$ and $N_\theta\in \mathcal{N}_\theta$. Then 
  for any $(t,\varepsilon)\in {\Omega}$ and
  $x\in K(t,\varepsilon)$
  \[
  N_\theta(K(t,\varepsilon))\subset K(n_\theta(t,\varepsilon))
  \subset K,\qquad \Norm{N_\theta(x)-x}\leq t_+-t,
  \]
 where $t_+$ is the first component of $n_\theta(t,\varepsilon)$.
  Moreover,
 \begin{equation} 
    \label{eq:ist} 
    n_\theta\left( {\Omega}\right)\subset {\Omega}, 
    \qquad N_\theta\left( K\right)\subset  K.
  \end{equation} 
\end{proposition} 
\begin{proof}
Combine definitions 
\eqref{eq:nftheta}, \eqref{eq:dt}, Definition \ref{def:inire},
 \eqref{eq:ker}, \eqref{eq:kt} 
 with  Lemmas \ref{lm:bas} and \ref{lm:id.t}.
\end{proof}
\section{Convergence analysis}\label{sec:cinmre}
\begin{theorem}
  \label{th:gc.ki.r}
  Take $0\leq\theta\leq\tildetheta$ and
  $N_\theta\in\mathcal{N}_\theta$.  For any
  $(t_0,\varepsilon_0)\in{\Omega}$ and $y_0\in K(t_0,\varepsilon_0)$
  the sequences
  \begin{equation}
    \label{eq:seq.02}
   y_{k+1}=N_\theta(y_k)  ,
\qquad
  ({t}_{k+1},\varepsilon_{k+1})=n_\theta({t}_k,\varepsilon_k) , \qquad k=0,1,\dots,
  \end{equation}
  are well defined,
  \begin{equation}
    \label{eq:in.02} 
  y_k\in K({t}_k,\varepsilon_k)   ,
    \qquad({t}_k,\varepsilon_k)\in {\Omega} \qquad k=0,1,\dots,
  \end{equation}
  the sequence $\{{t}_k\}$ is strictly increasing and converges to some
  ${\tilde  t}\in (0, \lambda ]$,
  the sequence $\{\varepsilon_k\}$ is non-decreasing and converges to
  some $\tilde \varepsilon\in [0, \kappa\lambda ]$,
  \begin{equation}
    \label{eq:bound.02} 
    \|F'(x_0)^{-1}F(y_k)\|\leq
    f({t}_k)+\varepsilon_k\leq
    \left(\frac{1+\theta^2}{2}\right)^{k}(f({t}_0)+\varepsilon_0),
    \quad k=0,1,\dots.
  \end{equation} 
  the sequence $\{y_k\}$ is contained in $B(x_0,\lambda )$ and
  converges to a point $x_*\in B[x_0,t_* ]$ which is the unique zero
  of $F$ in $B(x_0,\bar\tau)$ and
  \begin{equation}
    \label{eq:conv.02}
    \|y_{k+1}-y_k\|\leq {t}_{k+1}-{t}_k, \qquad \|x_{*}-y_k\|\leq
  {\tilde  t}-{t}_k, \qquad k= 0,1, \ldots \,.
  \end{equation}
  Moreover, if   
\begin{itemize}
  \item[{\bf h4')}] $\lambda<R$,
  \end{itemize}
 then  the sequence $\{y_k \}$ satisfies, for $k= 0,1,\ldots \,$
\begin{equation} \label{eq:qcyk}
\|x_*-y_{k+1}\|\leq \left[ \frac{1+\theta}{2}\frac{D^{-}f'(\lambda)}{|f'(\lambda)|}\|x_*-y_k\|
+\theta\,\frac{2+f'(\lambda)}{|f'(\lambda)|}\right]\|x_*-y_k\|.
\end{equation}
If, additionally, $0\leq \theta <\kappa/(4+\kappa)$ then  $\{y_k \}$ converges $Q$-linearly as  follows
  \begin{equation} \label{eq:qlyk}
  \norm{x_*-y_{k+1}}\leq \left[\frac{1+\theta}{2}+\frac{2\theta}{\kappa} \right]\norm{x_*-y_k}, \qquad k= 0,1, \ldots \,.
  \end{equation}
\end{theorem}
\begin{proof}
  Well definition of the sequences $\{(t_k,\varepsilon_k)\}$ and
  $\{y_k\}$ as defined in \eqref{eq:seq.02} follows from the
  assumptions on $\theta$, $(t_0,\varepsilon_0)$, $y_0$ and the two
  last inclusions on Proposition~\ref{pr:syn}.
  Moreover, since \eqref{eq:in.02} holds for $k=0$, using the first
  inclusion in Proposition~\ref{pr:syn} and induction on $k$, we
  conclude that \eqref{eq:in.02} holds for all $k$.
  The first inequality in \eqref{eq:conv.02} now follows from
  Proposition~\ref{pr:syn}, \eqref{eq:in.02} and \eqref{eq:seq.02}
  while the first inequality in \eqref{eq:bound.02} follows from
  \eqref{eq:in.02} and the definition of $K(t,\varepsilon)$ in
  \eqref{eq:ker}.

  Direct inspection of the definition of ${\Omega}$ in \eqref{eq:dt}
  shows that
  \[
   {\Omega}\subset [0,\lambda)\times [0,\kappa\lambda).
  \]
  Therefore, using \eqref{eq:in.02} and the definition of  $K(t,\varepsilon)$
  we have  
  \[
  t_k\in [0,\lambda),\quad \varepsilon_k\in [0,\kappa\lambda),
  \quad y_k\in B(x_0,\lambda),\qquad k=0,1,\dots
  \]
  Using \eqref{eq:dt} and Lemma~\ref{lm:id.t} we conclude that
  $\{t_k\}$ is strictly increasing, $\{\varepsilon_k\}$ is
  non-decreasing and the second equality in~\eqref{eq:bound.02} holds
  for all $k$. Therefore, in view of the first two above inclusions,
  $\{t_k\}$ and $\{\varepsilon_k\}$ converge, respectively, to some
  $\tilde t\in (0,\lambda]$ and $\tilde \varepsilon\in
  [0,\kappa\lambda]$. Convergence to $\tilde t$, together with the
  first inequality in \eqref{eq:conv.02} and the inclusion $y_k\in
  B(x_0,\lambda)$ implies that $y_k$ converges to some $x_*\in
  B[0,\lambda]$ and that the second inequality on \eqref{eq:conv.02}
  holds for all $k$.

  Using the inclusion $y_k\in B(x_0,\lambda)$, the first inequality in 
  Corollary~\ref{cr:new.01} and \eqref{eq:bound.02} we have
  \begin{equation*}
    -f(\Norm{y_k-x_0})\leq
  \left(\frac{1+\theta^2}{2}\right)^{k}(f({t}_0)+\varepsilon_0),
    \quad k=0,1,\dots.
  \end{equation*}
  According to \eqref{eq:bd.pr.b}, $f'<-\kappa$ in $[0,\lambda)$. Therefore,
  since $f(t_*)=0$ and $t_*<\lambda$,
  \[
  f(t)\leq -\kappa(t-t_*),\qquad t_*\leq t < \lambda.
  \]
  Hence, \emph{if} $\Norm{y_k-x_0}\geq t_*$, we can combine the two above
  inequalities, setting $t=\Norm{y_k-x_0}$ in the second, to obtain
  \[
  \Norm{y_k-x_0}-t_*\leq  
  \left(\frac{1+\theta^2}{2}\right)^{k}\frac{ f({t}_0)+\varepsilon_0}{\kappa}.
  \]
  Note that the above inequality remain valid even if $\Norm{y_k-x_0}<
  t_*$.  Therefore, taking the limit $k\to \infty$ in the above
  inequality we conclude that $\Norm{x_*-x_0}\leq t_*$. Moreover, now
  that we know that $x_*$ is in the interior of the domain of $F$, we
  can also take the limit $k\to\infty$ in \eqref{eq:bound.02} to
  conclude that $F(x_*)=0$. 

  The ``classical''  version of Kantorovich's theorem on Newton's method
  for a generic majorant function (see e.g. \cite{FerreiraSvaiter2009})
    guarantee that
  under the assumptions of Theorem~\ref{th:ki.r}, $F$ has a unique
  zero in $B(x_0,\bar\tau)$. Hence $x_*$ must be this zero of $F$.
  
To prove \eqref{eq:qcyk} and \eqref{eq:qlyk} , first note that from first inclusion in \eqref{eq:in.02}  we have $\|y_k-x_0\|\leq t_k$, for all $k= 0,1, \ldots \,.$ Now, since $\tilde t\in (0, \lambda ]$ we obtain from  second inequality in \eqref{eq:conv.02} that
 $ \|x_{*}-y_k\|\leq \lambda-{t}_k$,  for all $k= 0,1, \ldots \,.$ Therefore, using {\bf h4'}, $F(x_*)=0$ and  and first equality in \eqref{eq:seq.02}, the desire inequalities follows by applying  Lemma~\ref{pr:cq1}.  For concluding the proof,  note that for $0\leq \theta <\kappa/(4+\kappa)$  the quantity in the bracket in \eqref{eq:qlyk} is less than one, which implies that the sequence $\{y_k \}$ converges $Q$-linearly. 
\end{proof}
\begin{proposition}
  \label{pr:kappa}
  If $0\leq \rho<\beta/2$ then
  \[
  \rho<\bar t/2 <\bar t,\qquad f'(\rho)<0.
  \]
\end{proposition}
\begin{proof}
  Assumption $\rho<\beta/2$ and Proposition~\ref{pr:maj.f} item iii
  proves the first two inequalities of the proposition. The last
  inequality follows from the first inequality and
  Proposition~\ref{pr:maj.f} item i.
\end{proof}
\begin{proof}[Proof of Theorem~\ref{th:ki.r}]
First we will prove Theorem~\ref{th:ki.r} with $\rho=0$ and $z_0=x_0$. Note
that, from the definition in \eqref{eq:dktt.0}, we have
\[
 \kappa_0=\kappa,\quad \lambda_0=\lambda,\quad \Theta_0=\Theta.
\]
Since 
\[ (0,0)\in {\Omega}, \qquad x_0\in K(0,0),
\]
using Theorem~\ref{th:gc.ki.r} we conclude that Theorem~\ref{th:ki.r}
holds for $\rho=0$.

For proving the general case, take
\begin{equation}
	\label{eq:hyp.f}
	0\leq \rho<\beta/2,\qquad z_0\in B[x_0,\rho]\;.
\end{equation}
Using Proposition~\ref{pr:kappa} and \eqref{eq:def.bart} we conclude that
$\rho<\bar t/2$ and $f'(\rho)<0$. Define
\begin{equation}
  \label{eq:maj.01}
  g:[0,R-\rho)\to \mathbb{R}, \quad
  g(t)=\frac{-1}{f'(\rho)}[f(t+\rho)+2\rho].
\end{equation}
We claim that $g$ is a majorant function for $F$ at point $z_0$.
Trivially, $B(z_0,R-\rho)\subset C$, $g'(0)=-1$, $g(0)>0$. Moreover
$g'$ is also convex and strictly increasing.  To end the proof that
$g$ satisfies {\bf h1}, {\bf h2} and {\bf h3}, using
Proposition~\ref{pr:maj.f} item iii and second inequality in
\eqref{eq:hyp.f} we have
\[ 
\lim_{t\to \bar t-\rho}g(t)=\frac{-1}{f'(\rho)}(2\rho-\beta)<0\;.
\]
Using Proposition~\ref{pr:banach} we have
\begin{equation}
  \label{eq:maj.02}
  \Norm{F'(z_0)^{-1}F'(x_0)}\leq \frac{-1}{f'(\rho)}\;.
\end{equation}
Therefore, using also the second inequality of
Corollary~\ref{cr:new.01} we have
\begin{align*}
  \Norm{F'(z_0)^{-1}F(z_0)}\leq &  \Norm{F'(z_0)^{-1}F'(x_0)}
  \Norm{F'(x_0)^{-1}F(z_0)}\\
  \leq & \frac{-1}{f'(\rho)} [f(\Norm{z_0-x_0}) +2\Norm{z_0-x_0}].
\end{align*}
As $f'\geq -1$, the function $t\mapsto f(t)+2t$ is (strictly)
increasing.  Combining this fact with the above inequality and
\eqref{eq:maj.01} we conclude that
\[
\Norm{F'(z_0)^{-1}F'(z_0)}\leq g(0).
\]
To end the proof that $g$ is a majorant function for $F$ at $z_0$,
take $x,y\in \banacha$ such that
\[
x,y\in  B(z_0,R-\rho),\qquad \Norm{x-z_0}+\Norm{y-x}<
R-\rho\;. 
\]
Hence $x,y\in B(x_0,R)$,
$\Norm{x-x_0}+\Norm{y-x}<R$ and using \eqref{eq:maj.02} together with \eqref{MC.2} we have
\begin{align*}
  \|F'(z_0)^{-1}\left[F'(y)-F'(x)\right]\| \leq&
  \|F'(z_0)^{-1}F'(x_0)\| \|F'(x_0)^{-1}\left[F'(y)-F'(x)\right]\|\\
  \leq &\frac{-1}{f'(\rho)}\left[f'(\|y-x\|+\|x-x_0\|)-f'(\|x-x_0\|)\right].
\end{align*}
Since $f'$ is convex, the function $t\mapsto f'(s+t)-f'(s)$ is
increasing for $s\geq 0$ and $\|x-x_0\|\leq
\norm{x-z_0}+\norm{z_0-x_0}\leq \norm{x-z_0}+\rho$,
\begin{multline*}
 f'(\|y-x\|+\|x-x_0\|)-f'(\|x-x_0\|)\leq\\
  f'(\|y-x\|+\|x-z_0\|+\rho)
   -f'(\|x-z_0\|+\rho).
\end{multline*}
Combining the two above inequalities with the definition of $g$ we obtain
\[
\|F'(z_0)^{-1}\left[F'(y)-F'(x)\right]\|\leq 
  g'(\|y-x\|+\|x-z_0\|)-g'(\|x-z_0\|).
\]

Note that for $\kappa_\rho$, $\lambda_\rho$ 
and $\Theta_\rho$ as defined in \eqref{eq:dktt}, we have
\[
  \kappa_\rho=\sup_{0<t<R-\rho}\frac{-g(t)}{t},\quad
   \lambda_\rho =\sup \{t\in [0,R-\rho): \kappa_\rho+g'(t)<0\},
  \quad \Theta_\rho=\frac{\kappa_\rho}{2-\kappa_\rho},
\]
which are the same as \eqref{eq:dktt} with $g$ instead of
$f$. Therefore, applying Theorem~\ref{th:ki.r} for $F$ and the
majorant function $g$ at point $z_0$ and $\rho=0$, we conclude that
the sequence $\{z_k\}$ is well defined, remains in
$B(z_0,\lambda_\rho)$, satisfies \eqref{eq:th.rt} and converges to
some $z_*\in B[z_0,t_{*,\rho}]$ which is a zero of $F$, where
$t_{*,\rho}$ is the smallest solution of $g(t)=0$. Using
\eqref{eq:maj.01} we conclude that $t_{*,\rho}$ is the smallest
solution of
\[
 f(\rho+t)+2\rho=0.
\]
Hence, in view of Proposition~\ref{pr:maj.f} item ii, we have
$\rho+t_{*,\rho}< \bar t\leq \bar \tau$.  and
$B[(z_0,t_{*,\rho}]\subset B(x_0,\bar \tau)$. Therefore, $z_*$ is the
unique zero of $F$ in $B(x_0,\bar \tau)$, which we already called
$x_*$. 
Since
\[
g'(t)=f'(t+\rho)/|f'(\rho)|,\;\;D^{-}g'(t)=D^{-}f'(t+\rho)/|f'(\rho)|,
\qquad t\in [0,R-\rho),
\]
applying again  Theorem~\ref{th:ki.r} for $F$ and the
majorant function $g$ at point $z_0$ and $\rho=0$, we conclude
that item  {\bf h4} also holds.
\end{proof}
\section{Special cases} \label{sec:scinmer}
First we use Theorem~\ref{th:ki.r} to analyze the convergence of the inexact
Newton method with a relative residual error tolerance in the setting of
Smale's $\alpha$-theory. Up to our knowledge, this is the first time an inexact
Newton method with a relative error tolerance is analyzed in this framework.

\begin{theorem}
  \label{th:kngesrs}
  Let $\banacha$ and $\banachb$ be Banach spaces, $C\subseteq
  \banacha$ and $F:{C}\to \banachb$ a continuous function and analytic
  $int(C)$.  Take $x_0\in \mathrm{int}(C )$ with $F '(x_0)$
  non-singular.  Define
  \[ 
  \gamma := \sup _{ n > 1 }\left\| \frac {F'(x_0)^{-1}F^{(n)}(x_0)}{n
      !}\right\|^{1/(n-1)}.
  \]
  Suppose that $B(x_0, 1/\gamma)\subset C$,  $b>0$ and
  that
  \[
  \|F'(x_0)^{-1}F(x_0)\|\leq b,\qquad b\gamma<3-2\sqrt{2}, \qquad
  0\leq\theta\leq\frac{1-2\sqrt{\gamma b}-\gamma b}{1+2\sqrt{\gamma
      b}+\gamma b}\;.  
  \]
  Then, the sequence generated by the inexact Newton method for solving
  $F(x)=0$ with starting point $x_0$ and residual relative error
  tolerance $\theta$: For $k=0,1,\ldots,$
  \[ x_{k+1}=x_k+S_k,\quad
  \Norm{F'(x_0)^{-1}\left[F(x_k)+F'(x_k)S_k\right]}\leq \theta
  \|F'(x_0)^{-1}F(x_k)\|,
  \]
  is well defined, the generated sequence $\{x_k\}$ converges to a
  point $x_*$ which is a zero of $F$,
  \[ \
  \|F'(x_0)^{-1}F(x_k)\|\leq\left(\frac{1+\theta^2}{2}\right)^{k}b,
  \qquad k= 0,1, \ldots \,,
  \]
 the sequence $\{x_k\}$ is contained in $B(x_0, \lambda )$,
  $x_*\in B[x_0, t_* ]$ and $x_*$ is the unique zero of $F$ in
  $B(x_0,\bar\tau)$, where
  \[
  \lambda:=\frac{b}{\sqrt{\gamma b}+\gamma b},
  \]
  \[
  t_*= \frac{1+\gamma b-\sqrt{1-6\gamma b +(\gamma b)^2}}{4},
  \qquad \bar \tau= \frac{1+\gamma b+\sqrt{1-6\gamma b +(\gamma b)^2}}{4}\;.
  \]
   Moreover, the sequence $\{x_k \}$ satisfies, for $k= 0,1, \ldots \,$, 
\[
\|x_*-x_{k+1}\|\leq \left[ \frac{1+\theta}{2}\frac{D^{-}f'(\lambda)}{|f'(\lambda)|}\|x_*-x_k\|
+\theta\,\frac{f'(\lambda)+2}{|f'(\lambda)|}\right]\|x_*-x_k\|.
\]
  If, additionally, $0\leq \theta <(1-2\sqrt{\gamma b}-\gamma b)/(5-2\sqrt{\gamma b}-\gamma b)$ then  $\{x_k \}$ converges $Q$-linearly as  follows
  \[
  \norm{x_*-x_{k+1}}\leq \left[\frac{1+\theta}{2}+\frac{2\theta}{1-2\sqrt{\gamma b}-\gamma b} \right]\norm{x_*-x_k}, \qquad k= 0,1,
  \ldots \,.
  \]
\end{theorem}

\begin{proof}
  Since the function  $f:[0,1/\gamma) \to \mathbb{R}$ 
  \[
   f(t)=\frac{t}{1-\gamma t}-2t+b,
  \] 
  is a majorant function for $F$ in $x_0$, \cite{FerreiraSvaiter2009}. Therefore, 
  all results follow from Theorem~\ref{th:ki.r}, applied to this
  particular context.
\end{proof}

A semi-local convergence result for Newton method is instrumental in the
complexity analysis of linear and quadratic minimization problems by means of
self-concordant functions \cite{NesterovNemirovskii1994}.  Also in this
setting, Theorem~\ref{th:ki.r} provides a semi-local convergence result for
Newton method with a relative error tolerance.

\begin{theorem}
  \label{th:scinmes-scr}
  Let $C\subseteq \mathbb{R}^n$ be an open convex set and let $g:C \to
  \mathbb{R}$ be an $a$-self-concordant function with parameter $a>0$.
   For $x\in C$, let
   \[
  \Norm{v}_{x}:=\sqrt{v^Tg''(x)v}, \quad v\in \mathbb{R}^n,
   \]
  \[
  W_r(x):=\{ z: \Norm{z-x}_{x}<r\},\quad W_r[x]:=\{ z:
  \Norm{z-x}_{x}\leq r\}.
  \]
  Suppose that $x_0\in C$, $g''(x_0)$ is non-singular, $b>0$ 
  \[ 
  \|g''(x_0)^{-1}g'(x_0)\|_{x_0}\leq b<3-2\sqrt{2},
   \qquad 0\leq\theta\leq  \frac{1-2\sqrt{b}-b}{1+2\sqrt{ b}+b}\;.
  \]  
  Then the sequence generated by the inexact Newton method for solving
  $g'(x)=0$ with starting point $x_0$ and residual relative error
  tolerance $\theta$: For $k=0,1,\ldots,$
  \[
  x_{k+1}=x_k+S_k,\quad
  \Norm{g''(x_0)^{-1}\left[g'(x_k)+g''(x_k)S_k\right]}_{x_0}\leq \theta
  \|g''(x_0)^{-1}g'(x_k)\|_{x_0},
  \]
  is well defined, converges to a point $x_*$
  which is the (unique, global) minimizer of $g$, 
  \[  \|g''(x_0)^{-1}g'(x_k)\|_{x_0}\leq\left(\frac{1+\theta^2}{2}\right)^{k}b,
  \qquad k= 0,1, \ldots \,,
  \]
 the sequence $\{x_k\}$ is contained in $W_{\lambda}(x_0)$
  and $x_*\in W_{t_*}(x_0)$, where
 \[ 
 \lambda:=\frac{b}{\sqrt{b}+b},\qquad t_*= \frac{1+\ b-\sqrt{1-6\ b +b^2}}{4}.
  \]
   Moreover, the sequence $\{x_k \}$ satisfies, for $k= 0,1, \ldots \,$, 
\[
\|x_*-x_{k+1}\|\leq \left[ \frac{1+\theta}{2}\frac{D^{-}f'(\lambda)}{|f'(\lambda)|}\|x_*-x_k\|
+\theta\,\frac{f'(\lambda)+2}{|f'(\lambda)|}\right]\|x_*-x_k\|.
\]
  If, additionally, $0\leq \theta <(1-2\sqrt{b}-b)/(5-2\sqrt{b}-b)$ then  $\{x_k \}$ converges $Q$-linearly as  follows
  \[
  \norm{x_*-x_{k+1}}\leq \left[\frac{1+\theta}{2}+\frac{2\theta}{1-2\sqrt{b}-b} \right]\norm{x_*-x_k}, \qquad k= 0,1,
  \ldots \,.
  \]
\end{theorem}

\begin{proof} 
The scalar function $f:[0,1) \to
\mathbb{R}$ defined by
\[ f(t)=\frac{t}{1-t}-2t+b,
\] is a majorant function for $g'$ in $x_0$ , \cite{FerreiraSvaiter2009}. Therefore, the proof
follows from Theorem~\ref{th:ki.r}, applied to this particular context.
\end{proof}

\begin{theorem}
  \label{th:kngerl}
  Let $\banacha$ and $\banachb$ be a Banach spaces, $C\subseteq
  \banacha$ and $F:{C}\to \banachb$ a continuous function,
  continuously differentiable on $int(C)$.  Take $x_0\in
  \mathrm{int}(C )$ with $F '(x_0)$ non-singular.  Suppose that exist
  constants $L>0$ and $b>0$ such that $bL<1/2$, $B(x_0, 1/L)\subset C$
  and
  \[ \left\|F'(x_0)^{-1}\left[F'(y)-F'(x)\right]\right\| \leq L
  \|x-y\|,\qquad x,\, y\in B(x_0, 1/L),
  \]
  \[ \|F'(x_0)^{-1}F(x_0)\|\leq b,\qquad
  0\leq \theta\leq \frac{1-\sqrt{2bL}}{1+\sqrt{2bL}}
  \] 
  Then,  the sequence generated
  by the inexact Newton method for solving $F(x)=0$ with starting point
  $x_0$ and residual relative error tolerance $\theta$: For $k=0,1,\ldots,$
  \[
     x_{k+1}=x_k+S_k, \quad \Norm{F'(x_0)^{-1}\left[F(x_k)+F'(x_k)S_k\right]}\leq \theta
  \|F'(x_0)^{-1}F(x_k)\|,
  \]
  is well defined,
    \[ 
  \|F'(x_0)^{-1}F(x_k)\|\leq\left(\frac{1+\theta^2}{2}\right)^{k}b,
  \qquad k= 0,1, \ldots \,.
  \]
   the sequence $\{x_k\}$ is contained in
  $B(x_0, \lambda )$,  converges to a point  $x_*\in B[x_0, t_*]$
  which is the unique zero of $F$ in $B(x_0,1/L)$ where
  \[
  \lambda :=\frac{\sqrt{2bL}}{L},\quad
  t_*=\frac{1-\sqrt{1-2Lb}}{L}.
  \]
  Moreover,  the sequence $\{x_k \}$ satisfies, for $k= 0,1, \ldots \,$, 
\[
\|x_*-z_{k+1}\|\leq \left[ \frac{1+\theta}{2}\frac{L}{1-\sqrt{2bL}}\,\|x_*-x_k\|
+\theta\,\frac{1+\sqrt{2bL}}{1-\sqrt{2bL}}\right]\|x_*-x_k\|.
\]
  If, additionally, $0\leq \theta <(1-\sqrt{2bL})/(5-\sqrt{2bL})$ then the sequence  $\{x_k \}$ converges $Q$-linearly as  follows
  \[
  \norm{x_*-x_{k+1}}\leq \left[\frac{1+\theta}{2}+\frac{2\theta}{1-\sqrt{2bL}} \right]\norm{x_*-x_k}, \qquad k= 0,1,
  \ldots \,.
  \]
\end{theorem}

\begin{proof} 
Since the function  $f:[0,1/L)\to \mathbb{R},$
\[
 f(t):=\frac{L}{2}\;t^2-t+b,
\]
is a majorant function for $F$ at point $x_0$, all result follow from Theorem~\ref{th:ki.r}, applied to this particular context.
\end{proof}

\bibliographystyle{abbrv}

\begin{thebibliography}{10}

\bibitem{Alvarez2008}
F.~Alvarez, J.~Bolte, and J.~Munier.
\newblock A unifying local convergence result for {N}ewton's method in
  {R}iemannian manifolds.
\newblock {\em Found. Comput. Math.}, 8(2):197--226, 2008.

\bibitem{BlunCukerShubSmale1989}
L.~Blum, F.~Cucker, M.~Shub, and S.~Smale.
\newblock {\em Complexity and real computation}.
\newblock Springer-Verlag, New York, 1998.
\newblock With a foreword by Richard M. Karp.

\bibitem{Chen2006}
J.~Chen and W.~Li.
\newblock Convergence behaviour of inexact {N}ewton methods under weak
  {L}ipschitz condition.
\newblock {\em J. Comput. Appl. Math.}, 191(1):143--164, 2006.

\bibitem{Dembo1982}
R.~S. Dembo, S.~C. Eisenstat, and T.~Steihaug.
\newblock Inexact {N}ewton methods.
\newblock {\em SIAM J. Numer. Anal.}, 19(2):400--408, 1982.

\bibitem{FerreiraSvaiter2002}
O.~P. Ferreira and B.~F. Svaiter.
\newblock Kantorovich's theorem on {N}ewton's method in {R}iemannian manifolds.
\newblock {\em J. Complexity}, 18(1):304--329, 2002.

\bibitem{FerreiraSvaiter2009}
O.~P. Ferreira and B.~F. Svaiter.
\newblock Kantorovich's majorants principle for {N}ewton's method.
\newblock {\em Comput. Optim. Appl.}, 42(2):213--229, 2009.

\bibitem{HiriartLemarechal1993}
J.-B. Hiriart-Urruty and C.~Lemar{\'e}chal.
\newblock {\em Convex analysis and minimization algorithms. {I}}, volume 305 of
  {\em Grundlehren der Mathematischen Wissenschaften [Fundamental Principles of
  Mathematical Sciences]}.
\newblock Springer-Verlag, Berlin, 1993.
\newblock Fundamentals.

\bibitem{kantorovich1951}
L.~V. Kantorovi{\v{c}}.
\newblock The principle of the majorant and {N}ewton's method.
\newblock {\em Doklady Akad. Nauk SSSR (N.S.)}, 76:17--20, 1951.

\bibitem{kantAkil1964}
L.~V. Kantorovich and G.~P. Akilov.
\newblock {\em Functional analysis in normed spaces}.
\newblock Translated from the Russian by D. E. Brown. Edited by A. P.
  Robertson. International Series of Monographs in Pure and Applied
  Mathematics, Vol. 46. The Macmillan Co., New York, 1964.

\bibitem{Li2006}
C.~Li and J.~Wang.
\newblock Newton's method on {R}iemannian manifolds: {S}male's point estimate
  theory under the {$\gamma$}-condition.
\newblock {\em IMA J. Numer. Anal.}, 26(2):228--251, 2006.

\bibitem{Li2008}
C.~Li and J.~Wang.
\newblock Newton's method for sections on {R}iemannian manifolds: generalized
  covariant {$\alpha$}-theory.
\newblock {\em J. Complexity}, 24(3):423--451, 2008.

\bibitem{Li2009}
C.~Li, J.-H. Wang, and J.-P. Dedieu.
\newblock Smale's point estimate theory for {N}ewton's method on {L}ie groups.
\newblock {\em J. Complexity}, 25(2):128--151, 2009.

\bibitem{Moret1989}
I.~Moret.
\newblock A {K}antorovich-type theorem for inexact {N}ewton methods.
\newblock {\em Numer. Funct. Anal. Optim.}, 10(3-4):351--365, 1989.

\bibitem{Morini1999}
B.~Morini.
\newblock Convergence behaviour of inexact {N}ewton methods.
\newblock {\em Math. Comp.}, 68(228):1605--1613, 1999.

\bibitem{NesterovNemirovskii1994}
Y.~Nesterov and A.~Nemirovskii.
\newblock {\em Interior-point polynomial algorithms in convex programming},
  volume~13 of {\em SIAM Studies in Applied Mathematics}.
\newblock Society for Industrial and Applied Mathematics (SIAM), Philadelphia,
  PA, 1994.

\bibitem{Potra2005}
F.~A. Potra.
\newblock The {K}antorovich {T}heorem and interior point methods.
\newblock {\em Math. Program.}, 102(1, Ser. A):47--70, 2005.

\bibitem{Shen2009}
W.~Shen and C.~Li.
\newblock Kantorovich-type convergence criterion for inexact {N}ewton methods.
\newblock {\em Appl. Numer. Math.}, 59(7):1599--1611, 2009.

\bibitem{Smale1986}
S.~Smale.
\newblock Newton's method estimates from data at one point.
\newblock In {\em The merging of disciplines: new directions in pure, applied,
  and computational mathematics ({L}aramie, {W}yo., 1985)}, pages 185--196.
  Springer, New York, 1986.

\bibitem{Wang2009}
J.-H. Wang, S.~Huang, and C.~Li.
\newblock Extended {N}ewton's method for mappings on {R}iemannian manifolds
  with values in a cone.
\newblock {\em Taiwanese J. Math.}, 13(2B):633--656, 2009.

\bibitem{Wang2006}
J.-h. Wang and C.~Li.
\newblock Uniqueness of the singular points of vector fields on {R}iemannian
  manifolds under the {$\gamma$}-condition.
\newblock {\em J. Complexity}, 22(4):533--548, 2006.

\bibitem{Wang2007978}
J.-H. Wang and C.~Li.
\newblock Kantorovich's theorem for newton's method on lie groups.
\newblock {\em Journal of Zhejiang University: Science A}, 8(6):978--986, 2007.
\newblock cited By (since 1996) 0.

\bibitem{Ypma1984}
T.~J. Ypma.
\newblock Local convergence of inexact {N}ewton methods.
\newblock {\em SIAM J. Numer. Anal.}, 21(3):583--590, 1984.

\end{thebibliography}

\def\cprime{$'$}

\end{document}